\documentclass[a4paper,11pt]{amsart}
\usepackage{amssymb, amsmath, amsthm, mathtools, mathrsfs, bm}

\usepackage{tikz-cd}
\usepackage{extarrows}
\usepackage{arydshln} 
\usepackage{enumerate}

\usepackage{booktabs,float}
\usepackage[square, comma, sort, numbers]{natbib}
\usepackage[colorlinks = true,
            linkcolor = blue,
            urlcolor  = blue,
            citecolor = purple,
            anchorcolor = blue]{hyperref}

\numberwithin{equation}{section}
\theoremstyle{plain}
\newtheorem{theorem}{Theorem}[section]
\newtheorem{proposition}[theorem]{Proposition}
\newtheorem{corollary}[theorem]{Corollary}
\newtheorem{lemma}[theorem]{Lemma}

\theoremstyle{definition}

\theoremstyle{remark}

\newcommand{\xra}{\xlongrightarrow}

\newcommand{\CP}[1]{\mathbb{C}P^{#1}}

\newcommand{\p}{\mathfrak{P}}

\newcommand{\an}[1]{\ensuremath{\mathbf{A}_{n}^{#1}}}

\newcommand{\PD}{Poincar\'{e} duality\,}

\newcommand{\Z}{\ensuremath{\mathbb{Z}}}
\newcommand{\z}[1]{\ensuremath{\mathbb{Z}/2^{#1}}}

\DeclareMathOperator{\coker}{coker}
\DeclareMathOperator{\im}{im}

\DeclareMathOperator{\Hom}{Hom}

\DeclareMathOperator{\sq}{Sq}

\newcommand{\matwo}[2]{\ensuremath{\left(\begin{smallmatrix}
#1\\
#2
\end{smallmatrix}}\right)}

\title{Homotopy Types of Suspended $4$-manifolds}

\author{Pengcheng Li}

\address{School of Sciences, Great Bay University, Dongguan \rm{523000},  China}
\email{lipcaty@outlook.com}


\begin{document}

\begin{abstract}
  Given a closed, smooth, connected, orientable $4$-manifold $M$, whose integral homology groups can have $2$-torsion, we determine the homotopy decomposition of the double suspension $\Sigma^2M$ as wedge sums of some elementary $\mathbf{A}_3^3$-complexes, which are $2$-connected finite complexes of dimension at most $6$. Furthermore, we utilize the Postnikov square (or equivalently Pontryagin square) to find sufficient conditions for the homotopy decompositions of $\Sigma^2M$ to desuspend to that of $\Sigma M$.
\end{abstract}

\maketitle


\section{Introduction}\label{sec:intro}

Recently, research on the homotopy properties of manifolds has emerged in two directions. The first direction is the loop homotopy of manifolds, which can be traced back to Beben and Wu's work \cite{BW15} in 2011. After them, many people made efforts to promote the development of this project, such as Beben, Theriault and Huang \cite{BT14,TB22,HT22}. On the other hand, as exhibited by So and Theriault \cite{ST19},  the suspension homotopy of manifolds has rich applications in some important objects of geometry and physics, such as gauge groups and current groups. Hereafter, this research direction has been widely studied, such as \cite{Huang21,Huang22,Huang-arxiv,CS22,HL}. 

This paper contributes to a further research on the suspension homotopy of manifolds. In the above related literature, due to some intractable obstructions, the authors usually avoid to handle $2$-torsions of the integral homology groups of the manifolds. For example, So and Theriault \cite{ST19} required the $4$-manifolds are $2$-torsion-free in integral homology, Huang \cite{Huang-arxiv} restricts to $6$-manifolds with integral homology groups containing no $2,3$-torsions, while Cutler and So \cite{CS22}, Huang and Li \cite{HL} respectively studied the suspension homotopy of simply-connected $6$-manifolds and $7$-manifolds after localization away from $2$.

In this paper we developed new technique and tools in homotopy theory to obtain \emph{relatively complete} (see comments below Theorem \ref{thm:S2M}) characterizations of the homotopy types of suspended $4$-manifolds which can have $2$-torsion in homology. For instance, we successfully apply certain homotopy properties of some \emph{$\mathbf{A}_n^3$-complexes} (defined below) to obtain the homotopy decompositions of $\Sigma^2M$. Moreover, the Postnikov squaring operation (\ref{eq:PostSq}) and the Pontryagin squaring operation (\ref{eq:PontrSq}) appear to be powerful in the characterizations of the homotopy type of $\Sigma M$, see Section \ref{sect:proofs}.
 
To make sense of the Introduction, we need the following notions and notations. Let $G$ be an abelian group and let $n$ be a positive integer. Denote by $H_n(X;G)$ (resp., $H^n(X;G)$) the $n$-th (singular) homology (resp., cohomology) group of $X$ with coefficients in $G$, and denote by $P^n(G)$ the $n$-dimensional Peterson space (cf. \cite{Neisenbook}) which admits a unique non-trivial reduced integral cohomology group $G$ in dimension $n$. In particular, for integers $n,k\geq 2$, we denote by $\Z/k=\Z/k\Z$ the group of integers modulo $k$. Recall the Peterson space have the cell structure
\[ P^n(k)=P^n(\Z/k)=S^{n-1}\cup_k e^n,\] 
which admits the obvious inclusion $i_{n-1}$ of the bottom sphere $S^{n-1}$ into $P^n(k)$ and the pinch map $q_n$ onto $S^n$. 
For each $n\geq 3$,  there is a generator $\tilde{\eta}_r\in \pi_{n+1}(P^n(2^r))$ satisfying the formula
\[q_n\tilde{\eta}_r\simeq \eta_n,\] 
see Lemma \ref{lem:Moore-htpgrps}, where $\eta_n\colon S^{n+1}\to S^n$ is the iterated suspensions of the Hopf map $\eta\colon S^3\to S^2$.  For a homomorphism $\phi\colon G\to G'$ of groups, $\ker(\phi)$ and $\im(\phi)$ denote the kernel and the image subgroups of $\phi$, respectively.

A finite CW-complex $X$ is called an \emph{$\an{k}$-complex} if $X$ is $(n-1)$-connected and has dimension at most $n+k$. It is well-known that elementary (or called indecomposable) $\an{1}$-complexes consist of spheres $S^n,S^{n+1}$ and the Moore spaces $P^{n+1}(p^r)$ with $p$ odd primes and $r\geq 1$.  One may consult \cite{ZP17,lpc,ZLP19,ZP21,BH91} for more homotopy theory of such complexes. We need the following elementary $\an{3}$-complexes with $n\geq 3$ and $r,s\geq 1$: 
\begin{align*}
  C^{n+2}_\eta&=S^n\cup_\eta\bm C S^{n+1}=\Sigma^{n-2}\CP{2},\\ 
  C^{n+2}_r&=P^{n+1}(2^r)\cup_{i_{n}\eta}\bm C S^{n+1},\quad 
  C^{n+2,s}=S^n\cup_{\eta q_{n+1}}\bm C P^{n+1}(2^s),\\
  C^{n+2,s}_r&=P^{n+1}(2^r)\cup_{i_n\eta q_{n+1}}\bm C P^{n+1}(2^s);\\
  A^{n+3}(\eta^2)&=S^n\cup_{\eta^2}\bm C S^{n+2},\\
  A^{n+3}(\tilde{\eta}_r)&=P^{n+1}(2^r)\cup_{\tilde{\eta}_r}\bm C S^{n+2},\quad 
  A^{n+3}({2^r}\eta^2)=P^{n+1}(2^r)\cup_{i_n\eta^2}\bm C S^{n+2}.
\end{align*}
Here the first four $\an{2}$-complexes are the \emph{elementary Chang-complexes} (due to Chang \cite{Chang50}), and the last two spaces are the only two $\an{3}$-complexes with the homology groups:
\[H_n\cong\z{r},\quad H_{n+3}=H_0\cong\Z,\quad H_i=0 \text{ for $i\neq 0,n,n+3$}.\]
Compare \cite[Theorem 10.3.1]{Bauesbook}. Note that all of the above $\an{3}$-complexes desuspend: they can be defined for $n\geq 2$.

To deal with $2$-torsions in $H_\ast(M;\Z)$, we shall employ  the following cohomology operations. Let $X$ be a connected CW-complex. For each $r\geq 1$, there are unstable cohomology operations: the \emph{Postnikov square} 
\begin{equation}\label{eq:PostSq}
  \p_0\colon H^1(X;\z{r})\to H^3(X;\z{r+1})
\end{equation}
and the \emph{Pontryagin square} 
\begin{equation}\label{eq:PontrSq}
  \p_1\colon H^2(X;\z{r})\to H^4(X;\z{r+1}).
\end{equation}
These two operations were carefully studied by Whihtehead \cite{Whitehead50,Whitehead51}. Note that $\p_0$ is the \emph{suspension operation} of $\p_1$:
\begin{equation}\label{P0-P1}
  \sigma \p_0=\p_1\sigma,
\end{equation}
where $\sigma \colon H^\ast(X;G)\to H^{\ast+1}(\Sigma X;G)$ is the suspension isomorphism. 
The Adem relations 
\[\sq^3=\sq^1\sq^2,\quad 
\sq^3\sq^1+\sq^2\sq^2=0\]
yield the secondary operation $\Theta_n$ based on the relation $\varphi_n\theta_n=0$ with
\begin{equation}\label{eq:Theta-null}
  \begin{aligned}
 \theta_n=\matwo{\sq^2\sq^1}{\sq^2}&\colon K_n\to K_{n+3}\times K_{n+2},\\
 \varphi_n=(\sq^1,\sq^2)&\colon K_{n+3}\times K_{n+2}\to K_{n+4},
\end{aligned} 
\end{equation}
where $n\geq 1$, $K_m=K_m(\z{})$ denotes the Eilenberg-MacLane space of type $(\z{},m)$.  For each $r\geq 1$, the higher order Bockstein operations 
\begin{equation}\label{eq:highBockstein}
  \beta_r\colon H^\ast(X;\z{})\dashrightarrow H^{\ast+1}(X;\z{})
\end{equation}
 are inductively defined by setting $\beta_1$ as the usual Bockstein homomorphism associated to the short exact sequence \[0\to \z{}\to \Z/4\to \z{}\to 0;\]  
 for $r\geq 2$, $\beta_r$ is defined on the intersection of  $\ker(\beta_i)$, $i<r$, and takes values in the quotient by the $\im(\beta_i)$, $i<r$. This is also indicated by the dashed arrow in (\ref{eq:highBockstein}). See \cite[Section 5.2]{Harperbook} for more details. 
Note that the higher Bocksteins $\beta_r$ and the sequence $\Theta=\{\Theta_n\}_{n\geq 1}$ are both  \emph{stable} (cf. \cite[4.2.2]{Harperbook}): 
\[\Omega\beta_r=\beta_r,\quad \Omega \Theta_{n+1}=\Theta_n.\]

Let $M$ be a closed,  smooth, connected, orientable  $4$-manifold. By \PD and the universal coefficient theorem for cohomology, the homology groups $H_\ast(M;\Z)$ are given by the following table:
\begin{table}[H]
  \centering
  \begin{tabular}{cccccc}
    \toprule
    $i$& $0,4$& $1$&$2$&$3$&$\geq 5$\\
    \midrule
    $H_i(M;\Z)$&$\Z$&$\Z^m\oplus T$&$\Z^d\oplus T$&$\Z^m$&$0$\\
    \bottomrule
  \end{tabular}
  \caption{$H_\ast(M;\Z)$}\label{tab:M}
\end{table}
\noindent where $m,d$ are non-negative integers, and $T$ is a finitely generated torsion abelian group. Denote the $2$-primary component of $T$ by
\[T_2=\bigoplus_{j=1}^{n}\z{r_j}.\] 

Now it is prepared to state our first main theorem.

\begin{theorem}\label{thm:S2M}
  Let $M$ be a closed,  smooth, connected, orientable  $4$-manifold with integral homology $H_\ast(M;\Z)$ given by Table \ref{tab:M}. 
  \begin{enumerate}[1.]
    \item\label{S2M-spin} Suppose that $M$ is spin, then $\Sigma^2M$ has two possible homotopy types: 
    \begin{enumerate}
      \item If $\Theta\big(H^1(M;\z{})\big)=0$, then there is a homotopy equivalence 
      \[\Sigma^2M\simeq \big(\bigvee_{i=1}^m(S^3\vee S^5)\big)\vee \big(\bigvee_{i=1}^d S^4\big)\vee P^4(T)\vee P^5(T) \vee S^6.\]
      \item If $\Theta\big(H^1(M;\z{})\big)\neq 0$, then 
      \[\Sigma^2M\simeq \big(\bigvee_{i=1}^m(S^3\vee S^5)\big)\vee\big(\bigvee_{i=1}^d S^4\big)\vee P^4(\frac{T}{\z{r_{j_0}}})\vee P^5(T) \vee  A^6(2^{r_{j_0}}\eta^2), \]
      where $j_0$ is the maximum of the indices $j\leq n$ such that 
      \[\Theta(x)\neq 0,\beta_{r_j}(x)\neq 0,~x\in H^1(M;\z{}).\]
    \end{enumerate} 

    \item\label{S2M-nonspin} Suppose that $M$ is non-spin and $\Theta\big(H^1(M;\z{})\big)=0$, then $\Sigma^2M$ has three possible homotopy types:  
    \begin{enumerate}
      \item If for any $u\in H^4(\Sigma^2 M;\z{})$ with $\sq^2(u)\neq 0$ and any $v\in \ker(\sq^2)$, there hold
      \[
        \beta_r(u+v)=0,\quad u+v\notin \im(\beta_s),\quad \forall~r,s\geq 1,\]
        then there is a homotopy equivalence
        \[\Sigma^2M\simeq \big(\bigvee_{i=1}^m(S^3\vee S^5)\big)\vee \big(\bigvee_{i=1}^{d-1}S^4\big)\vee P^4(T)\vee P^5(T)\vee C^6_\eta.\]

     \item Suppose that for any $u\in H^2(M;\z{})$ with $\sq^2(u)\neq 0$ and any $v\in \ker(\sq^2)$, there hold 
     \[ u+v\notin \im(\beta_s),  ~\forall~s\geq 1,\] 
      while there exist $u'\in H^2(M;\z{})$ with $\sq^2(u')\neq 0$ and $v'\in \ker(\sq^2)$ such that 
      \[\beta_r(u'+v')\neq 0\text{ for some }r\geq 1.\]
     Then there is a homotopy equivalence 
\[\Sigma^2M\simeq \bigvee_{i=1}^m(S^3\vee S^5)\vee\bigvee_{i=1}^dS^4\vee P^4(T)\vee P^5(\frac{T}{\z{r_{j_1}}})\vee C^6_{r_{j_1}},\]
where $j_1$ is the maximum of the indices $j\leq n$ such that 
\[\sq^2(u')\neq 0,\quad \beta_r(u'+v')\neq 0.\]

\item Suppose that there exist $u\in H^2(M;\z{})$ with $\sq^2(u)\neq 0$ and $v\in \ker(\sq^2)$ such that \[u+v\in \im(\beta_r)\text{ for some $r$},\] 
then there is a homotopy equivalence
\[\Sigma^2M\simeq  \big(\bigvee_{i=1}^m(S^3\vee S^5)\big)\vee\big(\bigvee_{i=1}^dS^4\big)\vee P^4(\frac{T}{\z{r_{j_2}}})\vee P^5(T)\vee A^6(\tilde{\eta}_{r_{j_2}}),\]
where $j_2$ is the minimum of the indices $j\leq n$ such that $u+v\in \im(\beta_{r_j})$. 

    \end{enumerate}
  \end{enumerate}
\end{theorem}

We omit the discussion of the case where $M$ is non-spin and $\Theta$ acts non-trivially on $H^1(M;\z{})$. In this case, the double suspension of $M$ also has three possible homotopy types, which are similar to the second part (\ref{S2M-nonspin}) of Theorem \ref{thm:S2M}. The main difference is that there exists some index $j_0$, which is determined by $\Theta$ as in (\ref{S2M-spin}) and is usually different from $j_2$, such that $A^{6}(2^{r_{j_0}}\eta^2)$ is a wedge summand of $\Sigma^2 M$. Moreover, there maybe exists non-trivial Whitehead products,  one may need to consider $\Sigma^3M$ to  obtain a complete characterizations of the suspension homotopy of $M$.

We also study the homotopy type of the suspension $\Sigma M$ in terms of the Postnikov square $\p_0$ (or equivalently the Pontryagin square $\p_1$).

\begin{theorem}\label{thm:SM}
  Let $M$ be a closed,  smooth, connected, orientable  $4$-manifold with $H_\ast(M;\Z)$ given by Table \ref{tab:M}. If the Postnikov square 
  \[\p_0\colon H^1(M;\z{r_j})\to H^3(M;\z{r_j+1})\]
  is trivial for each $j=1,2,\cdots,n$, then the desuspensions of the homotopy decompositions of $\Sigma^2M$ in Theorem \ref{thm:S2M} yield the homotopy decompositions of $\Sigma M$.

\end{theorem}

If $H_\ast(M;\Z)$ contains no $2$-torsion (i.e., $T_2=0$), then the homotopy decomposition $\Sigma M\simeq \bigvee_{i=1}^mS^2\vee \Sigma W$ (\ref{eq:SW}) implies that the Pontryagin square \[\p_1\colon H^1(\Sigma M;\z{r_j})\to H^3(\Sigma M;\z{r_j+1})\] 
 is trivial, hence so is $\p_0$ by (\ref{P0-P1}). Hence Theorem \ref{thm:SM} extends So and Theriault's results \cite[Theorem 1.1]{ST19}.  However, the author didn't find any other $4$-manifolds $M$ satisfying conditions in Theorem \ref{thm:SM}. This is also why we arrange the above theorem after Theorem \ref{thm:S2M}.

The paper is organized as follows. In Section \ref{sec:telems} we review some homotopy theory of partial elementary $\mathbf{A}_n^3$-complexes and list some technical lemmas about the Pontryagin or Steenrod square operations. Section \ref{sec:AnalyMtd} introduces the main analysis methods adopted in this paper, including an useful criteria to determine the homtopy type of suspensions and the matrix method to determine the homotopy type of homotopy cofibres of certain maps. Section \ref{sec:homolydecomp} simply analyses the homology decomposition of the suspension $\Sigma M$. In Section \ref{sect:proofs} we utilize the methods developed in Section \ref{sec:AnalyMtd} to give a detailed discussion on the homotopy decompositions of our suspended four-manifolds. At the end, we prove Theorem \ref{thm:S2M} and \ref{thm:SM}, respectively.

\subsection*{Acknowledgements}
The author would like to thank Jianzhong Pan for some helpful discussion on Proposition \ref{prop:S2W-NSpin}. The author was partially supported by the National Natural Science Foundation of China (Grant no. 12101290).

\section{Some technical lemmas}\label{sec:telems}

 In this section we recall some homotopy groups of mod $2^r$ Moore spaces and prove some lemmas about the Pontryagin or Steenrod square operations. 

 Throughout, all spaces $X,Y,\cdots$ are based connected CW-complexes, and $[X,Y]$ is the set of based homotopy classes of based maps from $X$ to $Y$. We identify a map $f$ with its homotopy class in notation. For composable maps $g$ and $f$,  denote by $gf$ or $g\circ f$ the composition of $g$ with $f$. Unless otherwise specified,  $\bm C X$ denotes the reduced mapping cone of a space $X$, and $C_f$ denotes the homotopy cofibre of a given map $f\colon X\to Y$. For a cyclic group $G$, $G\langle x\rangle$ means $x$ is a generator of $G$.

\subsection{Some homotopy theory of mod $2^r$ Moore spaces}\label{sec:Moore} 
Let $n,k\geq 2$.
There is a homotopy cofibration for the mod $k$ Moore space $P^n(k)$:
\[S^{n-1}\xra{k}S^{n-1}\xra{i_{n-1}}P^n(k)\xra{q_n}S^n,\]
where $i_{n-1}$ and $q_n$ are the canonical inclusion and projection, respectively. Recall that if $2$ doesn't divide $k$, then 
\[\pi_n(P^n(k))=\pi_{n+1}(P^n(k))=0,~\forall~n\geq 3.\]
For each $r,s\geq 1$, let $\rho_r\colon \Z\to \z{r}$ be the reduction mod $2^r$ with $1_r=\rho_r(1)$, let $\chi^r_s\colon \z{r}\to \z{s}$ be the homomorphism given by 
\begin{equation}\label{eq:chi-grp}
  \chi^r_s(1_r)=\left\{\begin{array}{ll}
  1_s&r\geq s;\\
  2^{s-r}1_s&r<s.
\end{array}\right.
\end{equation}
For each $n\geq 3$, there exists a map (with $n$ omitted in notation) 
\[B(\chi^r_s)\colon P^{n+1}(2^r)\to P^{n+1}(2^s)\] 
 such that 
 \[H_n(B(\chi^r_s))=\chi^r_s, \quad \Sigma B(\chi^r_s)=B(\chi^r_s).\]
Moreover,  $B(\chi^r_s)$ satisfies the  relation formulas (cf. \cite{BH91}):
  \begin{equation}\label{eq:chi}
    \begin{aligned}
      B(\chi^r_s)i_{n}=\left\{\begin{array}{ll}
    i_{n},&r\geq s;\\
    2^{s-r}i_{n},&r\leq s.
  \end{array}\right.& ~q_{n+1}B(\chi^r_s)=\left\{\begin{array}{ll}
    2^{r-s}q_{n+1},&r\geq s;\\
    q_{n+1},&r\leq s.
  \end{array}\right.
    \end{aligned}
  \end{equation}
  Note that a multiple $t\alpha$ (or written as $t\cdot \alpha$) of an element $\alpha\in\pi_k(X)$ coincides with the composite $\alpha\circ t$.


\begin{lemma}\label{lem:Moore-htpgrps}
  Let $r\geq 1$ and $n\geq 3$ be integers.
\begin{enumerate}
  \item $\pi_{n-1}(P^{n}(2^r))\cong\z{r}\langle i_{n-1}\rangle$.
  \item $\pi_3(P^3(2^r))\cong \z{r+1}\langle i_2\eta\rangle$, $\pi_{n+1}(P^{n+1}(2^r))\cong\z{}\langle i_3\eta \rangle$.
    \item There are isomorphisms 
   \[\pi_{n+1}(P^n(2^r))\cong\pi_{n+2}(P^{n+1}(2^r))\cong \left\{\begin{array}{ll}
    \Z/4\langle \tilde{\eta}_1\rangle, & r=1;\\
    \z{}\langle \tilde{\eta}_r \rangle \oplus \z{}\langle i_2\eta^2\rangle,&r\geq 2, 
   \end{array}\right.\]
     where $\tilde{\eta}_r$ satisfies the formulas 
  \begin{equation}\label{eq:eta_r}
    \tilde{\eta}_r=B(\chi^1_r)\tilde{\eta}_1,\quad q_n\tilde{\eta}_r=\eta,\quad \tilde{\eta}_1=2\eta^2 q_{n+2}.
  \end{equation}
 \item Dually, there are isomrphisms 
  \[\pi^{n}(P^{n+2}(2^r))\cong \left\{\begin{array}{ll}
    \Z/4\langle \bar{\eta}_1 \rangle,& r=1;\\
    \z{}\langle \bar{\eta}_r\rangle\oplus\z{}\langle \eta^2 q_{n+2}\rangle,&r\geq 2,
  \end{array}\right. \]  
  where $\bar{\eta}_r$ satisfies the formula 
  \[\bar{\eta}_ri_{n+1}=\eta_n,\quad 2\bar{\eta}_1=\eta^2 q_{n+2}.\] 

\end{enumerate}

\begin{proof}
  (1) The isomorphism holds by the Hurewicz theorem.

 (2) By \cite[bottom of page 19, top of page 20]{Bauesbook},  there hold 
 \[\pi_n(P^n(2^r))\cong\left\{\begin{array}{ll}
  \Gamma(\z{r})\cong\z{r+1},&n=3;\\
  \z{r}\otimes\z{}\cong\z{},&n\geq 4.
 \end{array}\right.\]
 Here $\Gamma(\z{r})$ is the Whitehead's quadratic group, see \cite{Bauesbook} or \cite{Whitehead50}. The composite $i_{n-1}\eta$ is clearly a generator of $\pi_n(P^n(2^r))$.

 (3) By \cite[Proposition 11.1.12]{Bauesbook}, $\pi_4(P^3(2^r))$ is isomorphic to the stable homotopy group $\pi_4^s(P^3(2^r))$, whose generators and the relations (\ref{eq:eta_r}) refer to \cite{BH91}.

(4) The isomorphisms and the relation formulas follow by (3) under the Spanier-Whihtehead duality:
\[\pi^{n}(P^{n+2}(2^r))\cong \pi_{n+2}(P^{n+1}(2^r)).\]

\end{proof}
\end{lemma}
For simplicity we still denote $\tilde{\eta}_r\colon S^{n+1}\to P^n(2^r)$ the iterated suspensions of the generator $\tilde{\eta}_r$ of $\pi_4(P^3(2^r))$.
Combining (\ref{eq:chi}) and (\ref{eq:eta_r}), we have 
\begin{corollary}\label{cor:chi-eta}
 Let $r,s\geq 1$. There hold relations 
  \[B(\chi^r_s)\tilde{\eta}_r=\left\{\begin{array}{ll}
    \tilde{\eta}_s,&s\geq r;\\
    2^{r-s}\tilde{\eta}_s,&s\leq r.
  \end{array}\right.\]
\end{corollary}

%

\subsection{Whitehead's quadratic functor}\label{subsect:quadratic}
Recall the \emph{Whitehead's quadratic functor} 
 \[\Gamma\colon \mathbf{Ab}\to \mathbf{Ab}\]
on the category $\mathbf{Ab}$ of abelian groups \cite{Whitehead50,Bauesbook2}. The functor $\Gamma$ is characterized by the following property: a function $\varphi\colon G\to G'$ between abelian groups is called \emph{quadratic} if $\varphi(x)=\varphi(-x)$ and the function $G\times G\to G'$ with $(x,y)\mapsto \varphi(x+y)-\varphi(x)-\varphi(y)$ is bilinear. For each abelian group $G$, there is a \emph{universal quadratic function} \[\gamma=\gamma_G\colon G\to \Gamma(G)\] such that 
for any quadratic function $\varphi\colon G\to G'$, there is a unique homomorphism $\varphi^\square\colon \Gamma(G)\to G'$ such that $\varphi=\varphi^\square\circ \gamma$. It follows that for homomorphism $\phi\colon G\to G'$, there is a unique induced homomorphism $\Gamma(\phi)\colon \Gamma(G)\to \Gamma(G')$ such that $\Gamma(\phi)\circ\gamma_G=\gamma_{G'}\circ \phi$. The universal quadratic function $\gamma=\gamma_G$ induces the bilinear pairing 
\begin{equation}\label{eq:[1,1]}
  [1,1]\colon G\otimes G\to \Gamma(G), [1,1](x,y)=\gamma(x+y)-\gamma(x)-\gamma(y).
\end{equation}

\begin{lemma}[cf. \cite{Bauesbook}]\label{lem:Gamma} 
  Let $G$ be an abelian group and let $n\geq 0$.
 \begin{enumerate}
  \item  For cyclic group $G=\Z/n$ we have 
  \[\Gamma(\Z/n)\cong \Z/(n^2,2n),\]
 where $\Z/0=\Z$ and $(n^2,2n)$ is the greatest common divisor. The group is generated by $\gamma(1_n)$ with $1_n=1+n\Z$. 
 
 \item For any $x\in G$, there holds $\gamma(nx)=n^2\gamma(x)$.
 \end{enumerate}
\end{lemma}

\subsection{Squaring operations}\label{subsect:SqOp}
For an abelian group $G$, the Pontryagin square 
\[\p_1\colon H^2(X;G)\to H^4(X;\Gamma(G))\]
 is a \emph{quadratic} function with respect to the cup product $\smallsmile$:
\begin{equation}\label{eq:quad}
  \begin{aligned}
    \p_1(-x)&=\p_1(x),\quad \p_1(nx)=n^2\p_1(x),\\
    \p_1(x+y)&=\p_1(x)+\p_1(y)+[1,1]_\ast(x\smallsmile y),
  \end{aligned}
\end{equation}
where $[1,1]_\ast$ is induced by the coefficient homomorphism (\ref{eq:[1,1]}).  The Pontryagin square is natural with respect to maps $X\to Y$ between spaces and with respect to homomorphisms $G\to G'$ between groups.

Let $X$ be $\mathbf{A}_2^2$-complex and let 
\[ C_4(X)\xra{d}C_3(X)\xra{d}C_2(X)\] be its cellular chain complex. Represent a cohomology class $x\in H^2(X;G)$ by a cocycle $\hat{x}\colon C_2(X)\to G$, which induces a unique homomorphism \[\tilde{x}\colon H_2(X)=C_2(X)/d C_3(X)\to G,\] and therefore a unique homomorphism 
\[\Gamma(\tilde{x})\colon \Gamma(H_2(X))\to \Gamma(G).\] 
By the universal coefficient theorem, there is an isomorphism
\[\mu\colon H^2(X;H_2(X))\xra{\cong} \Hom(H_2(X),H_2(X)).\] 
 Let $\iota_2\in H^2(X;H_2(X))$ be given such that $\mu(\iota_2)$ is the identity on $H_2(X)$. By \cite[Chapter I]{Bauesbook2} we know that the Pontryagin square 
\[\p_1\colon H^2(X;G)\to H^4(X;\Gamma(G))\] 
is completely determined by the \emph{Pontryagin element} 
\[\p_1(\iota_2)\in H^4\big(X;\Gamma(H_2(X))\big)\] in the sense that there holds an formula
\begin{equation*}\label{eq:PontrSq-uni}
  \p_1(x)=\Gamma(\tilde{x})_\ast(\p_1(\iota_2)),
\end{equation*} 
where $\Gamma(\tilde{x})_\ast$ is induced by the coefficient homomorphism.

Let $C_r(t\eta)$ be the homotopy cofibre of $t\cdot i_2\eta\colon S^3\to P^3(2^r)$,  $r\geq 1$ and $t\in\z{r+1}$. Note that $C_r(t\eta)$ is an $\mathbf{A}_2^2$-polyhedron and has the \emph{$\mathbf{A}_2^2$-form} 
 \begin{equation}\label{an2-form}
  f=(t\eta,2^r)\colon S^3\vee S^2\to S^2,
 \end{equation}
 i.e., $C_r(t\eta)$ is the homotopy cofibre of the attaching map $f$ between spheres. 

\begin{lemma}\label{lem:PontrSq}
  Let $t\in\z{r+1}$, $r\geq 1$. The Pontryagin square 
  \[\p_1\colon H^2(C_r(t\eta);\z{r})\to H^4(C_r(t\eta);\z{r+1})\] 
  is trivial if and only if $t=0$.
  \begin{proof}
    Let $\iota_2\in H^2(C_r(t\eta);\z{r})$ be the generator which corresponds to the identity on $H_2(C_r(t\eta))$.  By \cite[Chapter I, Proposition 7.6]{Bauesbook2} and the $\mathbf{A}_2^2$-form (\ref{an2-form}), the Pontryagin element $\p_1(\iota_2)$ is represented by the cocycle 
  \[t\cdot \Gamma(\rho_r)\gamma=\Gamma(\rho_r)(t\gamma)\colon \Z\xra{t\gamma}\Gamma(\Z)\xra{\Gamma(\rho_r)}\Gamma(\z{r}).\]
Note that $\Gamma(\rho_r)\gamma=\gamma \rho_r$ represents a generator of $H^4(C_r(t\eta);\Gamma(\z{r}))$, by the universal coefficient theorem. Then it follows by Lemma \ref{lem:Gamma} that $\p_1=0$ if and only if $t=0$. 
  \end{proof}
\end{lemma}

Recall that the Steenrod square 
\[\sq^2\colon H^n(-;\z{})\to H^{n+2}(-;\z{})\] 
is a stable cohomology operation such that $\sq^2(x)=x^2$ for any cohomology class $x$ of dimension $2$, cf. \cite[Section 4.L]{hatcherbook}.
\begin{lemma}[cf. \cite{ZP17}]\label{lem:StSq-Changcpx}
For any $n\geq 3$, the Steenrod square 
   \[\sq^2\colon H^n(C;\z{})\to H^{n+2}(C;\z{})\]
    is an isomorphism for each (elementary) Chang-complex $C$.
 \end{lemma}

\begin{lemma}\label{lem:StSq-C_r}
 For each $n\geq 2,r\geq 1$, the Steenrod square  
 \[\sq^2\colon H^{n+1}(A^{n+3}(\tilde{\eta}_r);\z{})\to H^{n+3}(A^{n+3}(\tilde{\eta}_r);\z{})\]
is an isomorphism.
\begin{proof}
  By (\ref{eq:eta_r}) there is a homotopy commutative diagram of homotopy cofibrations (in which rows and columns are all homotopy cofibrations): 
  \[\begin{tikzcd}
    \ast\ar[d]\ar[r]&S^n\ar[r,equal]\ar[d,"i_n"]&S^n\ar[d]\\
    S^{n+2}\ar[r,"\tilde{\eta}_r"]\ar[d,equal]&P^{n+1}(2^r)\ar[d,"q_{n+1}"]\ar[r]&A^{n+3}(\tilde{\eta}_r)\ar[d,"d"]\\
    S^{n+2}\ar[r,"\eta"]&S^{n+1}\ar[r]&C^{n+3}_\eta
  \end{tikzcd}\]
 It follows that $d^\ast\colon H^k(C^{n+3}_\eta;\z{})\to H^k(A^{n+3}(\tilde{\eta}_r);\z{})$ is an isomorphism for $k=n+1,n+3$.
The isomorphism in the lemma then follows by Lemma \ref{lem:StSq-Changcpx} and the commutative square
\[\begin{tikzcd}
  H^{n+1}(C^{n+3}_\eta;\z{})\ar[r,"\sq^2","\cong"swap]\ar[d,"d^\ast","\cong"swap]&H^{n+3}(C^{n+3}_\eta;\z{})\ar[d,"d^\ast","\cong"swap]\\
  H^{n+1}(A^{n+3}(\tilde{\eta}_r);\z{})\ar[r,"\sq^2"]&H^{n+3}(A^{n+3}(\tilde{\eta}_r);\z{})
\end{tikzcd}\]
\end{proof}
 \end{lemma}

\subsection{Higher order cohomology operations}\label{subsect:highCO}

Recall the secondary cohomology operations 
\begin{equation}\label{eq:Theta}
 \Theta_n\colon S_n(X)\to T_n(X),
\end{equation}
based on the relation $\varphi_n\theta_n=0$ (\ref{eq:Theta-null}), where  
 \begin{align*}
  S_n(X)&=\ker(\theta_n)=\ker(\sq^2)\cap \ker(\sq^2\sq^1)\\
  T_n(X)&=\coker(\Omega\varphi_n)=H^{n+3}(X;\z{})/\im(\sq^1+\sq^2). 
 \end{align*}

 \begin{lemma}\label{lem:eta2}
Let $n\geq 2,r\geq 1$. For $X=A^{n+3}(\eta^2)$ or $A^{n+3}(2^r\eta^2)$, the secondary operation $\Theta_n$ acts non-trivially on $H^n(X;\z{})$; that is,
  \[0\neq\Theta_n \colon H^n(X;\z{})\to H^{n+3}(X;\z{}).\]

  \begin{proof}
    For $X=A^{n+3}(\eta^2)$ or $A^{n+3}(2^r\eta)$, we compute that
    \begin{align*}
     S_n(X)&=H^n(X;\z{})\cong\z{},\\
     T_n(X)&=H^{n+3}(X;\z{})\cong\z{}. 
    \end{align*}
    The proof of $\Theta_n\neq 0$ for $X=A^{n+3}(\eta^2)$ refers to \cite[page 96]{Harperbook}. 
   There is a homotopy cofibration 
   \[S^n\xra{i_n\circ 2^r}A^{n+3}(\eta^2)\xra{j}A^{n+3}(2^r\eta),\]
   which induces the following commutative square
    \[\begin{tikzcd}
      H^n(A^{n+3}(2^r\eta);\z{})\ar[d,"j^\ast", "\cong"swap]\ar[r,"\Theta_n"]&H^{n+3}(A^{n+3}(2^r\eta);\z{})\ar[d,"j^\ast","\cong"swap]\\
      H^n(A^{n+3}(\eta^2);\z{})\ar[r,"\Theta_n\neq 0"]&H^{n+3}(A^{n+3}(\eta^2);\z{})
    \end{tikzcd}\]
    Thus $\Theta\neq 0$ for $X=A^{n+3}(2^r\eta)$.
  \end{proof} 
 \end{lemma}

 The higher order Bocksteins (\ref{eq:highBockstein})
 \[\beta_r\colon H^n(X;\z{})\dashrightarrow H^{n+1}(X;\z{}) \] 
  are helpful to detect torsion elements of $H_\ast(X;\Z)$ or $H^\ast(X;\Z)$.
 
 \begin{lemma}[cf. \cite{MT68}, page 173 and 61]\label{lem:Bockstein}
 The following statements hold:
   \begin{enumerate} 
     \item The higher Bockstein $\beta_r$ detects the degree $2^r$ map on $S^n$; in other words,  for each $r\geq 1$, there are exactly one non-trivial higher Bockstein  
     \[\beta_r\colon H^{n-1}(P^n(2^r);\z{})\to H^n(P^n(2^r);\z{}).\]
     \item For each $r\geq 1$, elements of $H^\ast(X;\z{})$ coming from free integral homology class lie in $\ker(\beta_r)$ and not in $\im(\beta_r)$.
   \item If $z\in H^{n+1}(X;\Z)$ generates a direct summand $\z{r}$ for some $r$, then there exist generators  $z'\in H^n(X;\z{})$ and $z''\in H^{n+1}(X;\z{})$ such that 
     \[ \beta_r(z')=z'',\quad \beta_i(z')=\beta_i(z'')=0 \text{~for~} i<r.\]
   \end{enumerate}
 \end{lemma}

\section{Analysis methods}\label{sec:AnalyMtd}
In this section we list some auxiliary lemmas that simplify the proof arguments in the next section. These lemmas appear to be applicable to other similar problems as well, so we leave them in a separate section.

We say that a map $f\colon X\to Y$ is \emph{homologically trivial} if the induced  homomorphism $f_\ast\colon H_i(X)\to H_i(Y)$ is trivial for each $i$. 
\begin{lemma}[\cite{hatcherbook}, Theorem 4H.3]\label{lem:hlgdecomp}
  Let $X$ be a simply-connected space of dimension $N$. Write $H_i=H_i(X)$. Then there is a sequence $X_2\subseteq X_3\subseteq \cdots\subseteq X_m$ of subcomplexes $X_j$ of $X$ such that 
\begin{enumerate}[(1)]
  \item  $i_\ast\colon H_j(X_n)\cong H_j(X)$ for $j\leq n$ and $H_j(X_n)=0$ for $j>n$.
  \item $X_2=M_2(H_2)$, $X_N=X$.
  \item  There is a principal homotopy cofibration 
  \[M_n(H_{n+1})\xra{k_n} X_n\xra{i_n}X_{n+1}\to M_{n+1}(H_{n+1})\]
with $k_n$ homologically trivial.
\end{enumerate}

\end{lemma}

Note that we have the canonical inclusions $X^n\subseteq X_n\subseteq X^{n+1}$, where $X^k$ denotes the $k$-skeleton of $X$.
The map $k_n$ above is called the \emph{$n$-th $k'$-invariant}, and plays a key role in the homology decomposition of $X$. For instance, $k_n$ is null-homotopic if and only if $X_n\simeq X_{n-1}\vee M_n(H_nX)$.  

\begin{lemma}\label{lem:TrivChAct}
 Let $f\colon\vee_{i=1}^{m}A_i\to \vee_{j=1}^nB_j$ be a map which induces trivial homomorphism in cohomology groups with coefficients in abelian groups $G$ and $G'$. Let 
 \[f_{\jmath}=p_\jmath\circ f,\quad f_{\imath,\jmath}=f_{\jmath}\circ i_\imath=p_\jmath\circ f\circ i_\imath,\] 
 where $i_\imath\colon A_{\imath}\to \bigvee_{i=1}^{m}A_i$ and $p_\jmath\colon  \bigvee_{j=1}^nB_j\to B_\jmath$ are respectively the canonical inclusion and projection, $1\leq \imath\leq m,1\leq \jmath\leq n$. 
\begin{enumerate}
  \item\label{triv-cupprod} If $H^\ast(C_f;G)$ contains no non-trivial cup products, then so do $H^\ast(C_{f_{\jmath}};G)$ and $H^\ast(C_{f_{\imath,\jmath}};G)$, $\forall~\imath,\jmath$.
  \item\label{triv-CohOp} If the cohomology operation $\mathcal{O}\colon H^k(C_f;G)\to H^l(C_f;G')$
  is trivial, then so are the operations 
  \begin{align*}
    \mathcal{O}_\jmath&\colon H^k(C_{f_{\jmath}};G)\to H^l(C_{f_{\jmath}};G'),\\
    \mathcal{O}_{\imath\jmath}&\colon H^k(C_{f_{\imath\jmath}};G)\to H^l(C_{f_{\imath\jmath}};G').
  \end{align*}
  where $\mathcal{O}_\jmath$ and $\mathcal{O}_{\imath\jmath}$ are the cohomology operation of the same type as $\mathcal{O}$. 
\end{enumerate}

\begin{proof}
(1) The statement (\ref{triv-cupprod}) is due to \cite[Lemma 4.2]{ST19}. 

(2) By the proof of \cite[Lemma 4.2]{ST19}, for any integer $k\geq 0$ and any coefficient group $G$, there are monomorphisms 
\[d_\jmath^\ast\colon  H^k(C_{f_\jmath};G)\to H^k(C_f;G),\]
and epimorphisms 
\[d_{\imath\jmath}^\ast\colon H^k(C_{f_\jmath};G)\to H^k(C_{f_{\imath\jmath}^\ast};G).\]
Consider the following commutative diagrams:
\[\begin{tikzcd}
  H^k(C_f;G)\ar[d,"\mathcal{O}"]&\ar[l,"d_\jmath^\ast"swap,tail]H^k(C_{f_\jmath};G)\ar[d,"\mathcal{O}_\jmath"]\ar[r,"d_{\imath,\jmath}^\ast",two heads]&H^k(C_{f_{\imath,\jmath}};G)\ar[d,"\mathcal{O}_{\imath,\jmath}"]\\
  H^l(C_f;G')&\ar[l,"d_\jmath^\ast"swap,tail]H^l(C_{f_\jmath};G')\ar[r,"d_{\imath,\jmath}^\ast",two heads]&H^k(C_{f_{\imath,\jmath}};G')
\end{tikzcd}\]
It follows that $\mathcal{O}_\jmath$ is the restriction of $\mathcal{O}$, and $\mathcal{O}_{\imath\jmath}$ is induced by $\mathcal{O}_{\jmath}$. Thus if $\mathcal{O}$ is trivial, then so are $\mathcal{O}_{\jmath},\mathcal{O}_{\imath,\jmath}$.
\end{proof}
\end{lemma}

The following lemma is useful to determine the homotopy type of a suspension, see \cite[Lemma 6.4]{HL} or \cite[Lemma 5.6]{ST19}.
\begin{lemma}\label{lem:HL}
  Let $S\xra{f}(\bigvee_{i=1}^nA_i)\vee B \xra{g}\Sigma C$ be a homotopy fibration of simply-connected CW-complexes. Let $p_j\colon \bigvee_{i}A_i\to A_j$ be the canonical projection onto the wedge summand $A_j$, $j=1,\cdots, n$. Suppose that each composition
  \[f_j\colon S\xra{f}\bigvee_{i}A_i\xra{p_j}A_j\]
is null-homotopic, then there is a homotopy equivalence
\[\Sigma C\simeq \bigvee_{i=1}^nA_i\vee D,\]
where $D$ is the homotopy cofibre of the composition $S\xra{f}(\bigvee_{i}A_i)\vee B\xra{q_B}B$ with $q_B$ the obvious projection.
\end{lemma}

Let $X=\Sigma X'$, $Y_i=\Sigma Y_i'$ be suspensions, $i=1,2,\cdots,n$. Let \[i_l\colon Y_l\to \bigvee_{j=i}^nY_i,\quad p_k\colon \bigvee_{i=1}^n Y_i\to Y_k\] be respectively the canonical inclusions and projections, $1\leq k,l\leq n$. By the Hilton-Milnor theorem,  we may write a map
\(f\colon X\to\bigvee_{i=1}^nY_i\)
as \[f=\sum_{k=1}^n i_k\circ f_{k}+\theta,\]
where $f_{k}=p_k\circ f\colon X\to Y_k$ and $\theta$ satisfies $\Sigma \theta=0$. The first part $\sum_{k=1}^n i_k\circ f_{k}$ is usually represented by a vector: 
\[u_f=(f_1,f_2,\cdots,f_n)^t.\]
We say that $f$ is completely determined by its components $f_k$ if $\theta=0$; in this case, denote $f=u_f$. Let $h=\sum_{k,l}i_lh_{lk}p_k$ be a self-map of $\bigvee_{i=1}^nY_i$ which is completely determined by its components $h_{kl}=p_k\circ h\circ i_l\colon Y_l\to Y_k$. Denote by   
\[M_h\coloneqq (h_{kl})_{n\times n}=\begin{pmatrix}
  h_{11}&h_{12}&\cdots&h_{1n}\\
  h_{21}&h_{22}&\cdots&h_{2n}\\
  \vdots&\vdots&\ddots&\vdots\\
  h_{n1}&h_{n1}&\cdots&h_{nn}
\end{pmatrix}\] 
Then the composition law
\(h(f+g)\simeq h f+h g\)
implies that the product
\[M_h(f_1,f_2,\cdots,f_n)^t\]
given by the matrix multiplication represents the composite $h\circ f$.
Two maps $f=u_f$ and $g=u_g$ are called \emph{equivalent}, denoted by 
\[(f_1,f_2,\cdots,f_n)^t\sim (g_1,g_2,\cdots,g_n)^t,\]
if there is a self-homotopy equivalence $h$ of $\bigvee_{i=1}^n Y_i$, which can be represented by the matrix $M_h$, such that 
\[M_h(f_1,f_2,\cdots,f_n)^t\simeq (g_1,g_2,\cdots,g_n)^t.\] 
Recall that the above matrix multiplication refers to elementary row operations in matrix theory; and note that the homotopy cofibres of the maps $f=u_f$ and $g=u_g$ are homotopy equivalent if $f$ and $g$ are equivalent.

The following lemma serves as an example of the above matrix method.
\begin{lemma}\label{lem:S4P4}
 Define $X$ by the homotopy cofibration 
 \[S^4\xra{(f_1,f_2\cdots,f_n)^t}\bigvee_{j=1}^nV_j\xra{}X,\] 
 where $f_j\colon S^4\to V_j$, $j=1,\cdots,n$.
 \begin{enumerate}
   \item\label{S4P4-1} If $V_j=S^3$ for $j=1,2\cdots,n$ and $f_{j_0}=\eta$ for some $j_0$, then there is a homotopy equivalence  
    \[X\simeq C^5_\eta\vee \bigvee_{j\neq j_0}S^3.\]
   \item\label{S4P4-2} If $V_j=P^4(2^{r_j})$ for $j=1,2\cdots,n$, and $f_j=i_3\eta$ for some $j$, then there is a homotopy equivalence 
   \[X\simeq C^5_{r_{j_1}}\vee \bigvee_{j\neq j_1} P^4(2^{r_j}),\] 
where $j_1=\max\{1\leq j\leq n~|~f_j=i_3\eta\}$.
 \end{enumerate}

 \begin{proof}
   (1) If there are a unique $f_{j_0}=\eta$, the statement clearly holds. We may assume that $f_1=\eta$, $f_i=\varepsilon_i\cdot \eta$, $\varepsilon_i\in\{0,1\}$. Then 
   \[\begin{pmatrix}
     1&0&\cdots&0\\
     -\varepsilon_2&1&\cdots &0\\
     \vdots&\vdots&\ddots&  \vdots\\
     -\varepsilon_n&0&\cdots&1
   \end{pmatrix}\begin{pmatrix}
    \eta\\\varepsilon_2\cdot\eta\\\vdots\\
    \varepsilon_n\cdot\eta
   \end{pmatrix}\simeq \begin{pmatrix}
     \eta\\0\\ \vdots\\
     0
   \end{pmatrix}.\]
It follows that there exists a self-homotopy equivalence $e_S$ of $\vee_{j=1}^mS^3$ such that 
\[e_S f\sim (\eta,0,\cdots,0)^t,\] 
and hence there is a homotopy equivalence  
\[X=C_f\simeq C_{e_Sf}\simeq C^5_\eta\vee\bigvee_{j=2}^mS^3.\]

(2) The statement clearly holds if there is a unique $j$ such that $f_j=i_3\eta$. Let $j_1$ be defined in the lemma. If there is an index $j_2$ such that 
\[f_{j_2}=i_3\eta\in \pi_4(P^4(2^{r_{j_2}})),\]
then the matrix multiplication
\[\begin{pmatrix}
  1_P&0\\
  -B(\chi^r_s)&1_P
\end{pmatrix}\begin{pmatrix}
  i_3\eta\\
  i_3\eta
\end{pmatrix}\simeq \begin{pmatrix}
  i_3\eta\\
  0
\end{pmatrix}\]
implies that $(f_{j_1},f_{j_2})^t\sim (f_{j_1},0)^t$.  
 By induction, it follows that there exists a self-homotopy equivalence $e_P$ of $\bigvee_{j=1}^m P^4(2^{r_j})$ such that 
 \[e_P\circ (f_1,f_2,\cdots,f_n)^t\simeq (0,\cdots,0,i_3\eta,0,\cdots,0)^t,\]
 where $i_3\eta$ in the latter vector lies in the $j_1$-th position. Thus we have a homotopy equivalence
\[X=C_f\simeq  C^5_{r_{j_1}}\vee \bigvee_{j\neq j_1}P^4(2^{r_j}).\]
 \end{proof}
\end{lemma}

\section{Homology decomposition of $\Sigma M$}\label{sec:homolydecomp}

By (\ref{tab:M}) and \cite[Lemma 5.1]{ST19}, there is a homotopy equivalence
\begin{equation}\label{eq:SW}
  \Sigma M\simeq (\bigvee_{i=1}^mS^2)\vee \Sigma W,
\end{equation}
where $W$ is a CW-complex with integral homology given by the following table:
\begin{table}[H]
  \centering
  \begin{tabular}{cccccc}
    \toprule
    $i$&$0,4$& $1$&$2$&$3$&$\geq 5$\\
    \midrule
    $H_i(W)$&$\Z$&$T$&$\Z^d\oplus T$&$\Z^m$&$0$\\
    \bottomrule
  \end{tabular}
  \caption{$H_\ast(W;\Z)$}\label{tab:W}
\end{table}

By Lemma \ref{lem:hlgdecomp} and Table \ref{tab:W}, there are homotopy cofibrations
\begin{equation}\label{cof-k'inv}
  \begin{aligned}
     & \bigvee_{i=1}^dS^2\vee P^3(T)\xra{k_3}P^3(T)\to W_3,\\
  &\bigvee_{i=1}^mS^3\xra{k_4} W_3\to W_4,\quad S^5\xra{k_5}W_4\to \Sigma W,
  \end{aligned}
\end{equation}
where $k_3,k_4, k_5$ are homologically trivial maps.
Let $T_2=\bigoplus_{j=1}^{n}\z{r_j}$ be the $2$-primary component of $T$ and write $T=T_2\oplus T_{\neq 2}$. 
For each $k\geq 3$, there are homotopy equivalences (cf. \cite{Neisenbook})
\[P^k(T)\simeq P^k(T_2)\oplus P^k(T_{\neq 3})\simeq \big(\bigvee_{j=1}^{n}P^k(2^{r_j})\big)\vee P^k(T_{\neq 3}).\]

\begin{lemma}\label{lem:w_3}
  There is a homotopy equivalence 
  \[W_3\simeq \big(\bigvee_{i=1}^d S^3\big)\vee P^3(T)\vee P^4(T).\]
  \begin{proof}
    By (\ref{cof-k'inv}), there is a homotopy cofibration 
    \[\big(\bigvee_{i=1}^dS^2\big)\vee P^3(T)\xra{f}P^3(T)\to W_3,\] 
where $f$ is a homologically trivial map with its two components of the following types:
 \begin{align*}
  f_1^S&\colon \big(\bigvee_{i=1}^dS^2\big)\hookrightarrow \big(\bigvee_{i=1}^dS^2\big)\vee P^3(T)\xra{f}P^3(T);\\
  f_2^{T}&\colon P^3(T)\hookrightarrow \big(\bigvee_{i=1}^dS^2\big)\vee P^3(T)\xra{f}P^3(T).
 \end{align*}
 Here the arrows $``\hookrightarrow''$ denote the obvious inclusions. Clearly $f_1^S$ and $f_2^T$ are both homologically trivial.
Set $T=\bigoplus_{k=1}^lp_k^{r_k}$ with $p_k$ primes. Then the Hurewicz isomorphism $\pi_2(P^3(T))\cong H_2(P^3(T))$  implies that both $f_1^S$ and the composite   
\[S_{T}=\bigvee_{k=1}^lS^2\xra{j}P^3(T)\xra{f_2^T}P^3(T)\]
are null-homotopic,
where $j$ is the canonical inclusion.
Let 
\[m_T=\bigvee_kp_k^{r_k}\colon S_T\to S_T\]  be the attaching map of $P^3(T)$.
There is a homotopy commutative diagram of homotopy cofibrations (in which rows are columns are homotopy cofibrations):
\[\begin{tikzcd}
  S_T\ar[r]\ar[d,"m_T"]& \ast\ar[d]\ar[r]&\Sigma S_T\ar[d,"i_2\circ (\Sigma m_T)"]\\
  S_T\ar[r,"0"]\ar[d,"i"]&P^3(T)\ar[d,equal]\ar[r]& P^3(T)\vee \Sigma S_T\ar[d]\\
  P^3(T)\ar[r,"k_3^P"]&P^3(T)\ar[r]&C_{k_3^P}
\end{tikzcd}\]
It follows that
\[C_{k_3^P}\simeq P^3(T)\vee P^4(T),\]
and hence there is a homotopy equivalence 
\[W_3\simeq \bigvee_{i=1}^dS^3\vee C_{k_3^P}\simeq  \big(\bigvee_{i=1}^dS^3\big)\vee P^3(T)\vee P^4(T).\]
  \end{proof}
\end{lemma}

\begin{lemma}\label{lem:W_4}
  There is a homotopy equivalence 
	\[W_4\simeq \big(\bigvee_{i=1}^dS^3\big)\vee P^4(T)\vee C_{g_2}\]
	for some homologically trivial map $g_2\colon \bigvee_{i=1}^mS^3\to P^3(T)$. Moreover, there is a homotopy equivalence 
  \[\Sigma W_4\simeq \big(\bigvee_{i=1}^dS^4\big)\vee P^4(T)\vee P^5(T)\vee \bigvee_{i=1}^mS^5.\]
  \begin{proof}
   By (\ref{cof-k'inv}) and Lemma \ref{lem:w_3}, there is a homotopy cofibration  
\[\bigvee_{i=1}^mS^3\xra{g} \big(\bigvee_{i=1}^dS^3\big)\vee P^3(T)\vee P^4(T)\to W_4\]
with $g$ a homologically trivial map. $g$ is determined by the following components  
\begin{align*}
 g_1\colon & S^3\to\bigvee_{i=1}^mS^3\xra{g} \big(\bigvee_{i=1}^dS^3\big)\vee P^3(T)\vee P^4(T)\to \bigvee_{i=1}^dS^3\to S^3,\\
 g_2\colon &  S^3\to\bigvee_{i=1}^mS^3\xra{g} \big(\bigvee_{i=1}^dS^3\big)\vee P^3(T)\vee P^4(T)\to P^3(T),\\
 g_3\colon &S^3\to\bigvee_{i=1}^mS^3\xra{g} \big(\bigvee_{i=1}^dS^3\big)\vee P^3(T)\vee P^4(T)\to P^4(T).
\end{align*}
Here the unlabelled maps are the obvious inclusions and projections.
$g_1,g_2,g_3$ are all homologically trivial. The Hurewicz theorem implies that both $g_1$ and $g_3$ are null-homotopic. Then by Lemma \ref{lem:HL} we get the first statement. 

To prove the second homotopy equivalence, it suffices to show that if 
$f\colon S^4\to P^4(T)$ is homologically trivial, then $f$ is null-homotopic.  Consider the following homologically trivial components of $f$
\begin{align*}
  f_1\colon &S^4\xra{f}P^4(T)\to P^4(T_{\neq 2}),\\
  f_2^j\colon&S^4\xra{f}P^4(T)\to P^4(T_2)\to P^4(2^{r_j}),~j=1,2,\cdots,n.
\end{align*}
$f_1$ is clearly null-homotopic, because $\pi_4(P^4(p^r))=0$ for odd primes $p$.
Observe that $W_4=\Sigma W^4$ is a suspension, the Steenrod square $\sq^2$ acts trivially on $H^2(W_4;\z{})$. By Lemma \ref{lem:TrivChAct} (\ref{triv-CohOp}), $\sq^2$ acts trivially on $H^3(C_{f_2^j};\z{})$. Since $\pi_4(P^4(2^{r_j}))\cong \z{}\langle i_3\eta\rangle$ (Lemma \ref{lem:Moore-htpgrps}), we may set \[f_2^j=\varepsilon_j\cdot i_3\eta, \quad \varepsilon_j\in\z{}.\] 
Note that the homotopy cofibre of $i_3\eta\in\pi_4(P^4(2^{r_j}))$ is the  Chang-complex $C^5_{r_{j}}$, by Lemma \ref{lem:StSq-Changcpx} we then get that $\varepsilon_j=0$, or equivalently $f_2^j$ is null-homotopic for each $j=1,2,\cdots,n$. Thus $f$ is null-homotopic, by Lemma \ref{lem:HL}.
  \end{proof}
\end{lemma}

\section{Proofs of Theorem \ref{thm:S2M} and \ref{thm:SM}}\label{sect:proofs}

By (\ref{cof-k'inv}) and Lemma \ref{lem:W_4}, there is a homotopy cofibration 
\[S^5\xra{h}\big(\bigvee_{i=1}^dS^4\big)\vee P^4(T)\vee P^5(T)\vee \bigvee_{i=1}^mS^5\to \Sigma^2 W,\]
where $h$ a homologically trivial map, $T\cong T_2\oplus T_{\neq 2}$ with $T_2\cong \bigoplus_{j=1}^{n}\z{r_j}$. Since $\pi^5(P^4(p^r))=\pi_5(P^5(p^s))=0$ for any odd primes $p$, Lemma \ref{lem:HL} indicates that there is a homotopy equivalence 
\begin{equation}\label{eq:S2W}
  \Sigma^2 W\simeq P^4(T_{\neq 2})\vee P^5(T_{\neq 2})\vee  \big(\bigvee_{i=1}^mS^5\big)\vee C_\varphi,
\end{equation}
where $\varphi\colon S^5\to \big(\bigvee_{i=1}^dS^4\big)\vee P^4(T_2)\vee P^5(T_2)$ is a homologically trivial map. $\varphi$ has the following three types of components: 
\begin{align*}
  \varphi_1&\colon S^5\xra{\varphi}\big(\bigvee_{i=1}^dS^4\big)\vee P^4(T_2)\vee P^5(T_2)\to S^4,\\
  \varphi_2^j&\colon S^5\xra{\varphi} \big(\bigvee_{i=1}^dS^4\big)\vee P^4(T_2)\vee P^5(T_2)\to P^4(T_2)\to P^4(2^{r_i}),\\
  \varphi_3^j&\colon S^5\xra{\varphi}\big(\bigvee_{i=1}^dS^4\big)\vee P^4(T_2)\vee P^5(T_2)\to P^5(T_2)\to P^5(2^{r_i}),
\end{align*}
where $j=1,2,\cdots,n$ and the unlabelled maps are the obvious projections.

\begin{proposition}\label{prop:S2W-Spin}
 If $\sq^2\big(H^4(\Sigma^2 W;\z{})\big)=0$, then the homotopy type of $\Sigma^2 W$ is determined by the secondary operation $\Theta$ (\ref{eq:Theta}) and the higher Bockstein $\beta_r$.  Explicitly, if  $\Theta\big(H^3(C_\varphi;\z{})\big)=0$, then there is a homotopy equivalence
  \[C_{\varphi}\simeq  \big(\bigvee_{i=1}^dS^4\big)\vee P^4(T_2)\vee P^5(T_2)\vee S^6  ;\] 
otherwise we have
\[C_{\varphi}\simeq \big(\bigvee_{i=1}^dS^4\big)\vee P^4\big(\frac{T_2}{\z{r_{j_0}}}\big)\vee P^5(T_2)\vee  A^{6}(2^{r_{j_0}}\eta^2),  \]
where $j_0$ is the maximum of the indices $j$ satisfying 
\[\Theta(x)\neq 0,\beta_{r_j}(x)\neq 0,~x\in H^3(C_\varphi;\z{}).\]
  \begin{proof}
    By assumption and (\ref{eq:S2W}), $\sq^2$ acts trivially on $H^4(C_{\varphi};\z{})$, and hence so does $\sq^2$ on $H^4(C_{\varphi_1};\z{})$, $H^4(C_{\varphi_k^j};\z{})$ for each $k=2,3$ and $j=1,2,\cdots,n$, by Lemma \ref{lem:TrivChAct} (\ref{triv-CohOp}).
 It follows by Lemma \ref{lem:StSq-Changcpx} and \ref{lem:StSq-C_r} that $\varphi_1$, $\varphi_3^i$ are null-homotopic, and 
 \[\varphi_2^{j}=y_j\cdot i_3\eta^2,\quad y_j\in\z{},~j=1,2,\cdots,n.\] 
 By Lemma \ref{lem:eta2}, the coefficients $y_j$ can be detected by the secondary operation $\Theta$. There are possibly many such indices $j$, however, similar argumets to that in the proof of Lemma \ref{lem:S4P4} show that there exists a homotopy equivalence $e$ of $P^4(T_2)$ such that 
 \[e(\varphi_2^1,\varphi_2^2,\cdots,\varphi_2^{n})^t\simeq (0,\cdots,0,\varphi_2^{j_0}=i_3\eta^2,0,\cdots,0)^t\]
with $j_0$ described in the proposition. The proof then completes by 
applying Lemma \ref{lem:HL}.
  \end{proof}
\end{proposition}

\begin{proposition}\label{prop:S2W-NSpin}
  If $\sq^2\big(H^4(\Sigma^2 W;\z{})\big)\neq 0$ and $\Theta\big(H^3(C_\varphi;\z{})\big)=0$, then the homotopy type of $\Sigma^2 W$
can be characterized as follows.
\begin{enumerate}[1.]
  \item\label{S2W-1} Suppose that for any $u\in H^4(\Sigma^2 M;\z{})$ with $\sq^2(u)\neq 0$ and any $v\in \ker(\sq^2)$, there hold
  \[
    \beta_r(u+v)=0,\quad u+v\notin \im(\beta_s),\quad \forall~r,s\geq 1,\]
    then there is a homotopy equivalence 
    \[C_{\varphi}\simeq \big(\bigvee_{i=1}^{d-1}S^4\big)\vee P^5(T_2)\vee P^4(T_2)\vee C^6_\eta.\]
    \item\label{S2W-2} Suppose that for any $u\in H^4(\Sigma^2 M;\z{})$ with $\sq^2(u)\neq 0$ and any $v\in \ker(\sq^2)$, there hold 
    \[ u+v\notin \im(\beta_s),  ~\forall~s\geq 1,\] 
     while there exist $u'\in H^4(\Sigma^2 M;\z{})$ with $\sq^2(u')\neq 0$ and $v'\in \ker(\sq^2)$ such that 
     \[\beta_r(u'+v')\neq 0\text{ for some }r\geq 1.\]
    Then there is a homotopy equivalence 
    \[C_{\varphi}\simeq \big(\bigvee_{i=1}^dS^4\big)\vee P^5(\frac{T_2}{\z{r_{j_1}}})\vee P^4(T_2)\vee C^6_{r_{j_1}}\]
    with $j_1$ the maximum of indices $j$ such that  
    \[\sq^2(u')\neq 0, \beta_{r_{j_1}}(u'+v')\neq 0, u'\in H^4(\Sigma^2())  .\]
    
    \item\label{S2W-3} Suppose that there exist $u\in H^4(\Sigma^2 M;\z{})$ with $\sq^2(u)\neq 0$ and $v\in \ker(\sq^2)$ such that \[u+v\in \im(\beta_r)\text{ for some $r\geq 1$},\] 
     then there is a homotopy equivalence 
     \[C_{\varphi}\simeq \big(\bigvee_{i=1}^dS^4\big)\vee P^5(T_2)\vee P^4(\frac{T_2}{\z{r_{j_2}}})\vee A^6(\tilde{\eta}_{r_{j_2}})\]
     with $j_2$.
\end{enumerate}

\begin{proof}
 By the Hilton-Milnor theorem,  Lemma \ref{lem:Moore-htpgrps} and \ref{lem:eta2}, and the assumption that $\Theta\big(H^3(C_{\varphi};\z{})\big)=0$, we can put 
  \begin{equation}\label{eq:varphi}
    \varphi=\sum_{j=1}^d x_i\cdot \eta+\sum_{j=1}^{n}y_j\cdot i_4\eta+\sum_{k=1}^{n}z_{k}\cdot \tilde{\eta}_{r_{k}}+\theta,
  \end{equation}
  where $\theta$ is a linear combination of Whihtehead products in $\pi_5(P^4(T_2))$.

  By Lemma \ref{lem:StSq-Changcpx} and \ref{lem:StSq-C_r}, we see that at least one of these coefficients $x_i,y_j,z_k$ is non-zero.

  (1) Under  the conditions in (\ref{S2W-1}), we deduce from Lemma \ref{lem:Bockstein} that $u$ comes from a free integral homology class. It follows that 
  \[y_{j}=z_{k}=0,\text{ and } x_i=1 \text{~for some $i$}\]
  in the expression (\ref{eq:varphi}).
  By Lemma \ref{lem:S4P4} (\ref{S4P4-1}), we may assume that there is exactly one $j$ such that $x_j=1$. Thus by Lemma \ref{lem:HL}, we get the homotopy equivalence in (\ref{S2W-1}).
  
  (2) The arguments are similar to (1), the conditions (\ref{S2W-2}) implies that 
  \[x_i=z_k=0, \text{ and }y_j=1 \text{ for some }j,\] 
  while Lemma \ref{lem:S4P4} (\ref{S4P4-2}) guarantees that we may assume that there is exactly one such $j$, which equals to $j_1$ described in the proposition. The homotopy equivalence in (\ref{S2W-2}) then follows by Lemma \ref{lem:HL}.
  
  (3)  The conditions (\ref{S2W-3}) imply that 
  \[z_{k}\equiv 1 \pmod 2\text{ for some }k.\] 
   By Lemma \ref{lem:S4P4}, we may firstly assume that 
  \begin{align*}
    x_1=\epsilon\in\z{}, &\quad x_i=0\text{ for $i>1$};\\
  y_{j_0}=\varepsilon\in\z{},&\quad y_j=0\text{ for $j\neq j_0$}.
  \end{align*}
  Note that $(\epsilon,\varepsilon)\neq (1,1)$, because 
  \[\begin{pmatrix}
    \eta\\
    i_4\eta
  \end{pmatrix}\sim \begin{pmatrix}
    \eta\\
    0
  \end{pmatrix}\colon S^5\to S^4\vee P^5(2^r).\]
  By the relation $q_4\tilde{\eta}_{r_k}=\eta$ in (\ref{eq:eta_r}), we have  
  \[
   \begin{pmatrix}
     \eta\\
     \tilde{\eta}_{r_k}
   \end{pmatrix}\sim \begin{pmatrix}
    0\\
    \tilde{\eta}_{r_k}
  \end{pmatrix}, \quad \begin{pmatrix}
    i_4\eta\\
    \tilde{\eta}_{r_k}
  \end{pmatrix}\sim \begin{pmatrix}
   0\\
   \tilde{\eta}_{r_k}
  \end{pmatrix}.\]
  It follows that $z_k\equiv 1\pmod 2$ implies that $\epsilon=\varepsilon=0$.
   On the other hand, Corollary \ref{cor:chi-eta} implies that 
  \[\begin{pmatrix}
    \tilde{\eta}_r\\
    \tilde{\eta}_s
  \end{pmatrix}\sim \begin{pmatrix}
    \tilde{\eta}_r\\
    0
  \end{pmatrix} \text{ for $r\leq s$}.\] 
  Thus up to homotopy we may assume that $x_i=y_j=0$ and there exists exactly one $z_{k_0}\equiv 1\pmod 2$ with $k_0$ described in the proposition. 
  Then we get the homotopy equivalence in (\ref{S2W-3}) by Lemma \ref{lem:HL}.
\end{proof}
\end{proposition}

\begin{proof}[Proof of Theorem \ref{thm:S2M}]
  It is well-known that a closed,  smooth, connected, orientable $4$-manifold $M$ is spin if and only if the Steenrod square $\sq^2$ acts trivially on $H^2(M;\z{})$. The homotopy types of $\Sigma^2M$ in Theorem \ref{thm:S2M} then are obtained by (\ref{eq:SW}, \ref{eq:S2W}), Proposition \ref{prop:S2W-Spin} and \ref{prop:S2W-NSpin}.
\end{proof}

Next, we give a proof of Theorem \ref{thm:SM}. 
By (\ref{P0-P1}), there hold equivalence relations 
 \[\p_0\big(H^1(M;\z{r})\big)=0\iff \p_1\big(H^2(\Sigma M;\z{r})\big)=0.\]

\begin{lemma}\label{lem:trivPontrSq}
  If the Pontryagin square $\p_1$ acts trivially on $H^2(\Sigma M;\z{r})$, then so does $\p_1$ on $H^2(W_4;\z{r})$.
  \begin{proof}
    By Lemma \ref{lem:hlgdecomp} and the universal coefficient theorem for cohomology, the canonical inclusion $i\colon W_4\to \Sigma W$ induces isomophisms 
    \begin{align*}
      i^\ast&\colon H^2(\Sigma W;\z{r})\xra{\cong} H^2(W_4;\z{r}),\\
      i^\ast&\colon H^4(\Sigma W;\z{r+1})\xra{\cong} H^2(W_4;\z{r+1}).
    \end{align*}
If $\p_1$ acts trivially on $H^2(\Sigma M;\z{r})$, then so does $\p_1$ on $H^2(\Sigma W;\z{r})$, by (\ref{eq:SW}). The following commutative diagram 
\[\begin{tikzcd}
   H^2(\Sigma W;\z{r})\ar[d,"i^\ast","\cong"swap]\ar[r,"\p_1=0"] &H^4(\Sigma W;\z{r+1})\ar[d,"i^\ast", "\cong"swap]\\
  H^2(W_4;\z{r})\ar[r,"\p_1"]& H^4(W_4;\z{r+1})
\end{tikzcd}\]
then implies $\p_1=0$ on the second row.
  \end{proof}
\end{lemma}


 \begin{lemma}\label{lem:w4}
  If the Pontryagin square $\p_1$ acts trivially on $H^2(\Sigma M;\z{r_j})$ for each $j=1,2, \cdots,n$, then there is a homotopy equivalence
  \[W_4\simeq \big(\bigvee_{i=1}^d S^3\big)\vee \big(\bigvee_{i=1}^m S^4\big)\vee P^3(T)\vee P^4(T).\]
\begin{proof}
By Lemma \ref{lem:W_4} there is a homotopy equivalence 
  \[W_4\simeq  \big(\bigvee_{i=1}^dS^3\big)\vee P^4(T)\vee C_{g_2}\]
  for some homologically trivial map $g_2\colon \bigvee_{i=1}^mS^3\to P^3(T)$. It suffices to show the homologically trivial component 
\[g_2\colon S^3\to P^3(T)\]
is null-homotopic.
  By Lemma \ref{lem:HL}, it suffices to show that the components
  \begin{align*}
    g_2^j\colon S^3\xra{g_2} P^3(T_2)\xra{p_j}P^3(2^{r_j})
  \end{align*}
  are null-homotopic for each $j=1,2,\cdots,n$. 
    
Since $\pi_3(P^3(2^r))\cong\z{r+1}$,  we may set 
\[g_2^j=t_j\cdot i_2\eta \]
for some $t_j\in\z{r_j+1}$, $j=1,2,\cdots,n$.
The assumption and Lemma \ref{lem:trivPontrSq} imply that the Pontryagin square 
 \[\p_1\colon H^2(W_4;\z{r_j})\to H^4(W_4;\z{r_j+1})\]
is trivial. By the universal coefficient theorem for cohomology, $g_2^j$ induces trivial homomorphism in mod $2^{r_j}$ or mod $2^{r_j+1}$ cohomology, and hence by Lemma \ref{lem:TrivChAct} (\ref{triv-CohOp}), the Pontryagin square 
  \[\p_1\colon H^2(C_{g_2^j};\z{r_j})\to H^4(C_{g_2^j};\z{r_j+1})\]
  is trivial for each $j$. Then it follows by Lemma \ref{lem:PontrSq} that $t_j=0$, or equivalently $g_2^j$ is null-homotopic for each $j=1,2,\cdots,n$.
  \end{proof}
\end{lemma}

\begin{proof}[Proof of Theorem \ref{thm:SM}]
  By Lemma \ref{lem:w4} and (\ref{cof-k'inv}), there is a homotopy cofibration
  \[S^4\xra{k_5}W_4\simeq \big(\bigvee_{i=1}^d S^3\big)\vee \big(\bigvee_{i=1}^m S^4\big)\vee P^3(T)\vee P^4(T)\to \Sigma W\]
  with $k_5$ homologically trivial. Since $\pi_4(P^3(p^r))=\pi_4(P^4(p^r))=0$, Lemma \ref{lem:HL} implies that there is a homotopy equivalence
  \[\Sigma W\simeq \big(\bigvee_{i=1}^mS^4\big)\vee P^3(T_{\neq 2})\vee P^4(T_{\neq 2})\vee C_{\phi},\]
  where $\phi\colon S^4\to (\bigvee_{i=1}^dS^3) \vee P^3(T_2)\vee P^4(T_2)$ is a homologically trivial map. Compare (\ref{eq:S2W}). The discussion on the homotopy type of $\Sigma W$ is totally parallel to that of $\Sigma^2 W$ in the proofs of Proposition \ref{prop:S2W-Spin} and \ref{prop:S2W-NSpin}. The proof then completes by (\ref{eq:SW}). 
\end{proof}

\bibliographystyle{amsalpha}
\bibliography{refs}

\end{document}